\documentclass{amsproc}
\usepackage{amsmath}
\usepackage{amssymb}
\usepackage{thm-restate}
\usepackage[all,arc]{xy}
\usepackage{enumerate}
\usepackage{mathrsfs}
\usepackage{subcaption,graphicx}
\usepackage{pinlabel}
\newtheorem{thm}{Theorem}[section]

\newtheorem{prop}[thm]{Proposition}
\newtheorem{lem}[thm]{Lemma}

\usepackage{graphicx}
\theoremstyle{definition}
\newtheorem{defn}[thm]{Definition}

\newtheorem{exmp}[thm]{Example}

\usepackage{xcolor}
\theoremstyle{remark}
\newtheorem{rem}[thm]{Remark}

\newcommand{\Hom}{\text{Hom}}
\usepackage{mathrsfs}
\usepackage{pinlabel}
\usepackage{graphicx}

\usepackage{}

\newtheorem{theorem}{Theorem}[section]

\theoremstyle{definition}

\theoremstyle{remark}

\numberwithin{equation}{section}

\begin{document}

\title{Biquandles, quivers and virtual bridge indices}


\author{Tirasan Khandhawit}
\address{272 Rama VI Road, Thung Phayathai,
Ratchathewi, Bangkok, 10400}
\curraddr{}
\email{tirasan.kha@mahidol.edu}
\thanks{}

\author{Puttipong Pongtanapaisan}
\address{901 Palm Walk, Tempe, AZ 85281}
\curraddr{}
\email{ppongtan@asu.edu}
\thanks{}

\author{Brandon Wang}
\address{901 Palm Walk, Tempe, AZ 85281}
\curraddr{}
\email{bhwang10@asu.edu}
\thanks{}

\subjclass[2020]{Primary 57K10}

\date{}

\begin{abstract}
We investigate connections between biquandle colorings, quiver enhancements, and several notions of the bridge numbers \(b_i(K)\) for virtual links, where \(i=1,2\). We show that for any positive integers \(m \leq n\), there exists a virtual link \(K\) with \(b_1(K) = m\) and \(b_2(K) = n\), thereby answering a question posed by Nakanishi and Satoh. In some sense, this gap between the two formulations measures how far the knot is from being classical. We also use these bridge number analyses to systematically construct families of links in which quiver invariants can distinguish between links that share the same biquandle counting invariant.
\end{abstract}

\maketitle

\section{Introduction}
A \textit{biquandle} is an algebraic structure that generalizes the idea of a \textit{quandle}, which is in turn a generalization of conjugation in group theory. Moreover, its axioms can be seen as motivated by the Reidemeister moves in knot theory \cite{joyce1982classifying,matveev1982distributive}. One can associate a \textit{fundamental biquandle} to each knotted object type, which offers more information than the fundamental group in some situations. In fact, the fundamental biquandle completely classifies knots in 3-space, given a mild assumption. Even though it provides different insights into knot theory, biquandles can be difficult to understand. It is conjectured in \cite{fenn2004biquandles} that the fundamental biquandle is a complete invariant of virtual knots up to mirror reflection.

To unravel the structures of biquandles, it is common practice to study homomorphisms between them. A basic way to distinguish two biquandles, \( X \) and \( X' \), is to compare the cardinalities \( | \text{Hom}(X,Y) | \) and \( | \text{Hom}(X',Y) | \). In fact, fundamental tools in knot theory, such as Fox colorings, arise in this way. When $X$ is the fundamental biquandle, the authors of \cite{vo2024learning} showed that \( | \text{Hom}(X,Y) | \) can be used to show that the gap between the first and the second bridge index of virtual knots can be arbitrarily large. In this paper, we strengthen the result in \cite{vo2024learning}. In particular, this answers Question 3.2 in \cite{nakanishi2015two}.
\begin{restatable}{thm}{main}
\label{thm:main}
    For any positive integers $m \leq n$, there is a virtual knot $K$ with $b_1(K) = m$ and $b_2(K) = n$.
\end{restatable}

 Using the terminology in \cite{elhamdadi2015quandles}, we are also interested in \textit{enhancements}, which are stronger invariants that allow us to recover \( | \text{Hom}(X,Y) | \). Some enhancements include cocycle enhancements and subquandle polynomial enhancements (see \cite{elhamdadi2015quandles} and references therein for more examples). More recently, the quandle coloring quiver was introduced in \cite{cho2019quandle} as a categorification of the quandle counting invariant. In that paper, the authors already demonstrated that the quandle coloring quiver is a proper enhancement. That is, there exists a pair of links, \( L \) and \( L' \), such that the number of colorings of \( L \) by a quandle \( X \) is the same as that of \( L' \), but \( L \) and \( L' \) are distinguished by their quiver enhancement.

In addition to classical links, quandle coloring quivers have also been defined for other knotted objects, such as surface-links \cite{kim2021quandle} and virtual links \cite{nelson2024quandle}. Tanaguchi showed that when dihedral quandles of prime orders are used, the quiver never gives proper enhancements \cite{taniguchi2021quandle}. In this paper, we use intuitions from the connections between bridge numbers and biquandles to systematically decide when the enhancement is proper. For instance, in a similar vein to Taniguchi's theorem, we can conclude the following statement.

\begin{restatable}{lem}{bridgequiver}
\label{lem:main}
       If the biquandle counting invariant using a finite biquandle $|X|$ of a virtual link $L_i$ is exactly $|X|^{b_2(L_i)}$ for $i=1,2$, then the quiver enhancements for $L_1$ and $L_2$ are isomorphic. 
\end{restatable}

Some work has also been done on using the quiver to enhance other existing enhancements \cite{cho2019quandle2,istanbouli2020quandle}. By this, we mean that while the cocycle invariant is already an enhancement of the counting invariant, there exist links with identical cocycle invariants that are nevertheless distinguished by their quivers. Towards the end of this paper, we explore previously unexamined quiver enhancements of other enhancements.

\subsection*{Acknowledgements}
This work was partially supported by the AMS–Simons Travel Grant, which enabled the first and second authors to have many helpful conversations with Sam Nelson.

\section{Preliminaries}\label{sec:prelims}
\subsection{Biquandles}
In this section, we review the definitions of biquandles and relevant enhancements. 


\begin{defn}\label{def:biq}
    A \textit{biquandle} is a triple $(X,\overline{\triangleright}, \underline{\triangleright})$, where $X$ is a set and  binary operations satisfying the following properties.
\begin{enumerate}
    \item[(i)] For all $x\in X$, $x \overline{\triangleright} x=x\underline{\triangleright} x$,
    \item[(ii)]  The maps $\alpha_y, \beta_y : X \rightarrow X$ and $S : X \times X \rightarrow X \times X$ defined by $\alpha_y(x) = x \overline{\triangleright} y, \beta_y(x) = x \underline{\triangleright}  y$ and
$S(x, y) = (y \overline{\triangleright}x, x \underline{\triangleright} y)$ are invertible,
\item[(iii)] For all $x,y,z\in X$, $(x\overline{\triangleright} y)\overline{\triangleright} (z\overline{\triangleright}  y) = (x\overline{\triangleright} z)\overline{\triangleright}(y\underline{\triangleright} z)$ \\ $(x\underline{\triangleright} y)\overline{\triangleright} (z\underline{\triangleright}  y) = (x\overline{\triangleright}z)\underline{\triangleright} (y\overline{\triangleright} z)$ \\ $(x\underline{\triangleright}y)\underline{\triangleright}(z\underline{\triangleright} y) = (x\underline{\triangleright}z)\underline{\triangleright}(y\overline{\triangleright}  z)$,
\end{enumerate}
\end{defn}
\begin{exmp}
   A biquandle with operations  $x\overline{\triangleright} y=x$ is called a \textit{quandle}. In this case, one often uses the notation $x\underline{\triangleright} y= x\triangleright y$.
\end{exmp}
\begin{exmp}
    The \textit{dihedral quandle} of order $n$, denoted by $R_n$, is a set $\{1,2,\ldots, n \}$ equipped with the operation $x\triangleright y = 2y-x \mod n.$
\end{exmp}

For finite biquandles, it is often convenient to record a biquandle as a list of lists of lists. For instance, the list of lists of lists below represents a biquandle structure on the set $\{1,2,3,4\}$ where the blue list of lists keeps track of $\overline{\triangleright}$ and the black list of lists keeps track of $\underline{\triangleright}$. The first blue list means that $1\overline{\triangleright}1=2, 1\overline{\triangleright}2=3, 1\overline{\triangleright}3=1,$ and $1\overline{\triangleright}4=4$.

\begin{align*}
    (\textcolor{blue}{((2, 3, 1, 4), (3, 2, 4, 1), (4, 1, 3, 2), (1, 4, 2, 3))},\\
   ((3, 1, 2, 4), (4, 2, 1, 3), (2, 4, 3, 1), (1, 3, 4, 2)))
\end{align*}

Various examples of infinite biquandles arise as fundamental biquandles of knotted objects. We define a \textit{semiarc} to be an edge in the virtual link diagram where we consider each crossing as a 4-valent vertex.

\begin{defn}
   Given an oriented virtual link diagram $D$ representing the virtual link $L$, we label the semiarcs of $D$ as $x_1,x_2\cdots, x_n$. The \textit{fundamental biquandle} $X_L$ has a presentation with generators $x_1,\cdots x_n$ and relations as shown in Figure \ref{fig:quancross} for each crossing and vertex.
\end{defn}
\begin{figure}[ht!]
\labellist
\small\hair 2pt
\pinlabel $y$ at -1 63
\pinlabel $x$ at -6 3
\pinlabel $z=y\overline{\triangleright}x$ at 85 -3
\pinlabel $w=x\underline{\triangleright}y$ at 85 75
\pinlabel $x$ at 141 63
\pinlabel $y$ at 136 3
\pinlabel $z=x\underline{\triangleright}y$ at 225 -3
\pinlabel $w=y\overline{\triangleright}x$ at 225 75
\endlabellist
\includegraphics[width=9cm]{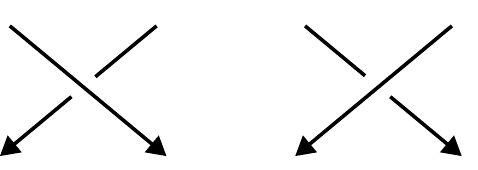}
\centering
\caption{(Left) Each crossing of negative type adds the relations $z=y\overline{\triangleright}x$ and $w=x\underline{\triangleright}y$. (Right) Each crossing of positive type adds the relations $z=y\overline{\triangleright}x$ and $w=x\underline{\triangleright}y$.}\label{fig:quancross}
\end{figure}

\begin{defn}
    Let $X, Y$ be biquandles. A map $f : X \rightarrow Y$ is a \textit{biquandle homomorphism} if for all $x, y \in X$, $f(x \underline{\triangleright} y) = f(x) \underline{\triangleright} f(y)$ and $f(x \overline{\triangleright} y) = f(x) \overline{\triangleright} f(y)$ . A biquandle \textit{endomorphism} is a biquandle homomorphism from a biquandle to itself.
\end{defn}

Let $\Hom(X,Y)$ denote the set of biquandle homomorphisms from $X$ to $Y$. We also denote $|\Hom(X_L,Y)|$ by $Col_Y(L)$. A labeling of $D$ by elements of a finite biquandle $Y$ such that the labelings near each crossing satisfy the rule in Figure \ref{fig:quancross} then represents an element of $\Hom(X,Y)$.

\subsection{Bridge indices} \label{subsec:bridge}

There are two natural ways to define the bridge index of virtual knot. An \textit{overpass} of a virtual knot diagram is a subarc that goes over at least one crossing
but never goes under a crossing. A \textit{maximal overpass} is an overpass that
could not be made any longer. The \textit{first bridge index} $b_1(K)$ of a virtual knot $K$ is the minimum number of maximal overpasses over all diagrams $D$ representing $K$. 

For the second definition, we consider the generic standard height function $h:\mathbb{R}^2\rightarrow\mathbb{R}$ restricted to the virtual knot diagram. The \textit{second bridge index} $b_2(K)$ of a virtual knot $K$ is the minimum number of local maxima for $h$ over all diagrams $D$ representing $K$. 

Motivated by \cite{blair2020wirtinger}, the second author proved that the first bridge index can be defined in another way \cite{pongtanapaisan2019wirtinger}.

A diagram $D$ is said to be \textit{partially colored} if we have specified a subset $A$ of the strands of $D$ and a function $f : A \rightarrow \{c_1,c_2,\cdots,c_n\}$. We refer to this partial coloring by $(A, f)$. We also say that two strands $s_i$ and $r_i$ of $D$ are \textit{adjacent} if $s_i$ and $r_i$ are the
understrands of some crossing in $D$. Given partial colorings $(A_1, f_1)$ and
$(A_2, f_2)$ of $D$, we say $(A_2, f_2)$ is the result of a \textit{coloring move} on $(A_1, f_1)$ if
\begin{enumerate}
    \item $A_1\subset A_2$ and $A_2 \backslash A_1$ is a collection of strands $\{s_1,s_2,\cdots,s_m\}$
    \item $\left.f_2\right|_{A_1}=f_1$
    \item For $j\in \{1,2,\cdots,m \},$ each $s_j$ is adjacent to a strand $r_j$ at crossings $c_j$, and $r_j \in A_1$.
    \item The overstrand $t_j$ at the crossing $c_j$ is an element of $A_1$.
    \item $f_2(s_j)=f_1(r_j)$.
\end{enumerate}
We sometimes refer to $(A_0, f_0) \rightarrow (A_1, f_1) \rightarrow \cdots \rightarrow (A_n, fn)$ as a \textit{Wirtinger coloring sequence}.
\begin{figure}[ht!]
\includegraphics[width=6cm]{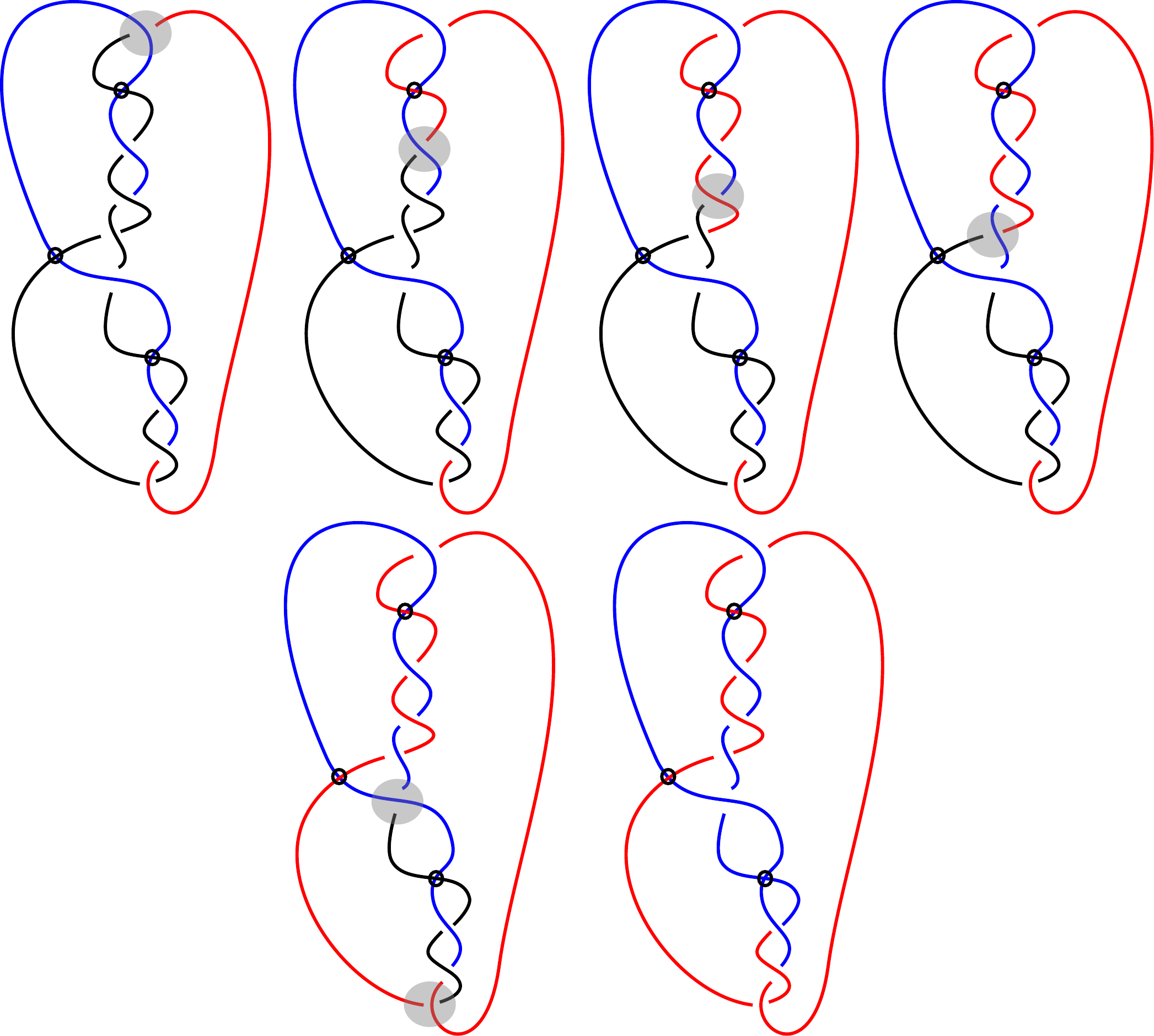}
\centering
\caption{A complete Wirtinger coloring sequence demonstrated on $\kappa_{2,0}$.}\label{fig:wirtseq}
\end{figure}

\begin{figure}[ht!]
\includegraphics[width=6cm]{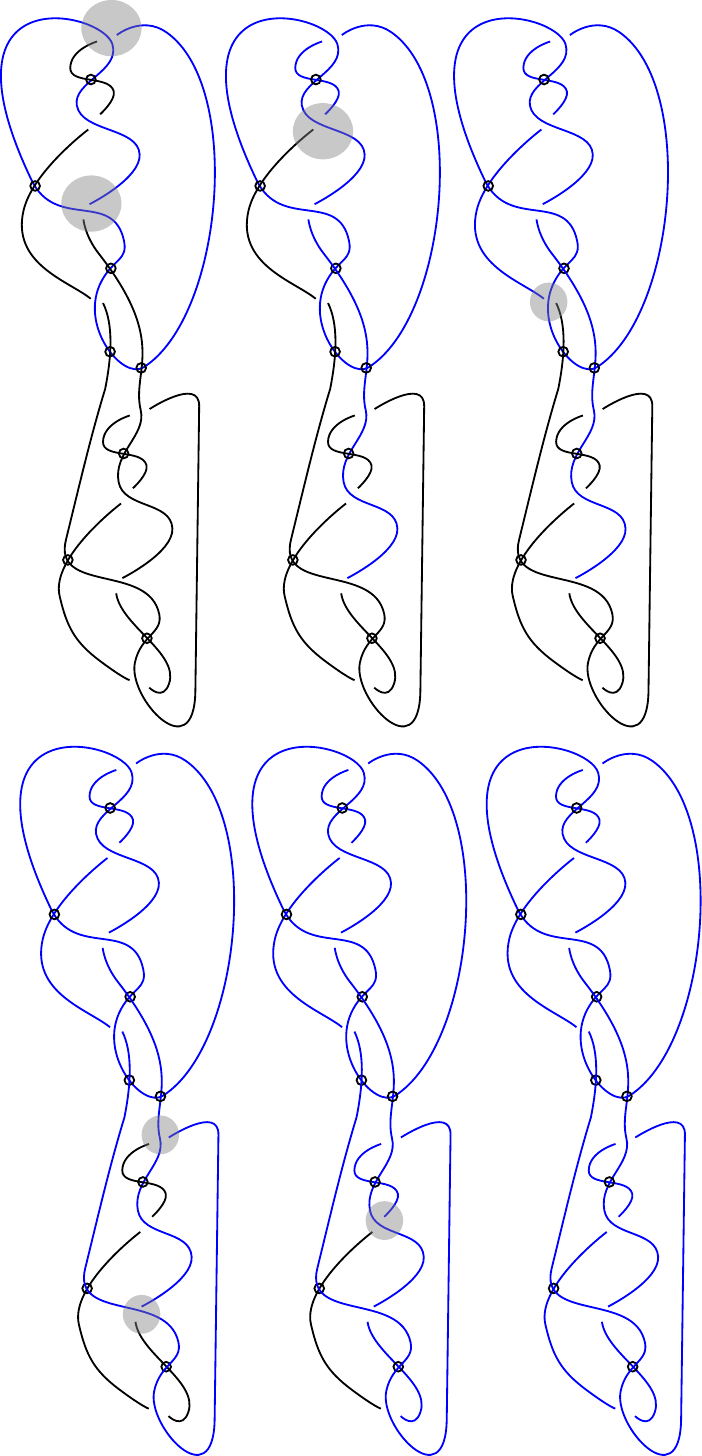}
\centering
\caption{The diagram has more than one overpasses, but one seed generates a complete Wirtinger coloring sequence, so $b_1(\kappa)=1$.}\label{fig:1seed}
\end{figure}

This process is demonstrated in Figure \ref{fig:wirtseq}. In the figure, there are six stages. In stage one, we have a diagram $D_1$, where two strands receive colors. The initial colored strands are also called the \textit{seeds}. The crossings where coloring moves are being performed are marked by transparent gray circles. Results in \cite{pongtanapaisan2019wirtinger} implies that $b_1(K)\leq 2$ in Figure \ref{fig:wirtseq}.

\subsection{Relating bridge numbers with quandles}

The following observation has been mentioned in various papers, but we provide the proof again here for completeness. 

\begin{prop}\label{prop:bridge}
    Let $X$ be a finite quandle. Then, $Col_X(L) = |X|^{b_1(L)}$, where $b_1(L)$ is the first bridge index of $L$.
\end{prop}

\begin{proof}
    Labelings of local maxima generate labelings of the entire knot diagram via the quandle coloring rule. The quantity $Col_X(L)$ counts the labelings that are consistent, which means that the rule $z=x\triangleright y$ in Figure \ref{fig:quancross} is satisfied at each crossing.
\end{proof}
Analogous arguments can be made for biquandles and the second bridge index. This relationship was explored in \cite{vo2024learning}.

\begin{prop}\label{prop:bridge2}
     Let $X$ be a finite biquandle. $Col_X(L) = |X|^{b_2(L)}$, where $b_2(L)$ is the second bridge index of $L$.
\end{prop}
We will also use the following fact which can be found in many existing works. See \cite{khandhawit2024quandle} for example. Recall that a collection of propertly embedded arcs $T$ in a 3-ball $B^3$ is called a \textit{trivial tangle} if it
contains no closed curve components and $T$ can be moved to arcs on $\partial B^3$
by an isotopy rel $\partial T$. It is well known that a two-string trivial tangle has an associated rational number $p/q.$ In fact, a \textit{rational tangle} is another name for a 2-string trivial tangle. It is also well-known that $|p|$ is equal to another invariant called the determinant.

\begin{prop}[Corollary 3.3 of \cite{khandhawit2024quandle}]
   If $L$ is a 2-bridge knot, then
   \begin{align*}
       |Hom(X_L,R_n)|=n\cdot gcd(\Delta,n), 
   \end{align*}
   where $\Delta$ is the determinant of the knot.\label{prop:propagate}
\end{prop}

Recall that a result of tying rational tangles in a circular fashion like a necklace is called a Montesinos link $M(p_1/q_1,\cdots,p_n/q_n)$. Figure \ref{fig:Monte} shows a Montesinos link $M(a/b,c/d,e/f)$ of length three. A \textit{pretzel link} is a Montesinos link where all the tangles involved are integral tangle.

\begin{figure}[ht!]
\includegraphics[width=6cm]{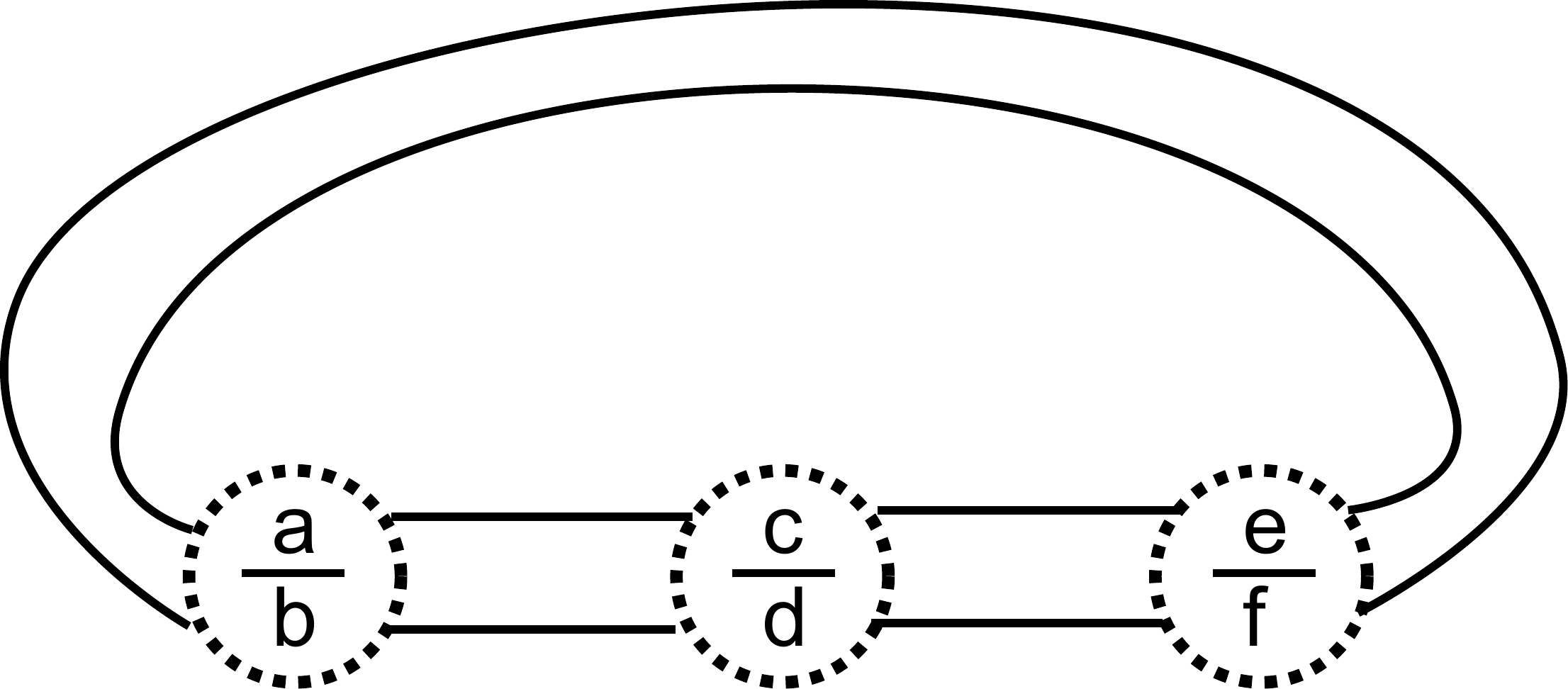}
\centering
\caption{A Montesinos link.}\label{fig:Monte}
\end{figure}

\begin{figure}[ht!]
\labellist
\small\hair 2pt
\pinlabel $3$ at 49 120
\endlabellist
\includegraphics[width=3cm]{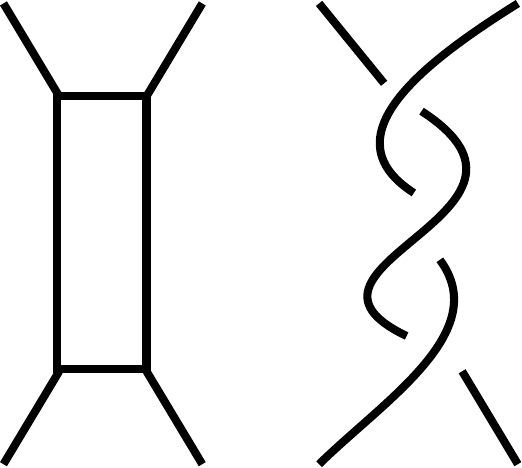}
\centering
\caption{A twist box means a string of crossings.}\label{fig:twist}
\end{figure}

\subsection{Enhancements}

\subsubsection{Biquandle coloring quivers}
\begin{defn}
    Given a knot $L$, its fundamental biquandle $X_L$, another finite biquandle $Y$, and a subset $S$ of $\Hom(Y,Y)$. The \textit{biquandle coloring quiver} $\mathcal{Q}^{S}_Y(L)$ is a directed graph whose vertices are elements of $\Hom(X_L,Y)$ and there is an edge from vertex $v$ to another vertex $w$ if there is an element $\varphi$ of $S$ such that $\varphi(v) = w.$
\end{defn}

Observe that \( \mathcal{Q}^S_X(L) \) is determined up to isomorphism by \( Y \) and \( \text{Hom}(X_L, Y) \). Hence, the biquandle coloring quiver is an invariant of \( L \). Comparing and contrasting two oriented graphs can still be challenging. A simpler way to distinguish two quivers in our setting is to record the degrees of each vertex. Furthermore, each vertex has the same number of edges oriented out of it, but the number of edges oriented into a vertex $v$, denoted by $deg^+(v)$ provides interesting data. This information can be systematically stored in a polynomial.

\begin{defn}
    Let $V$ denote the set of vertices for $\mathcal{Q}^S_
X(L)$. The \textit{in-degree polynomial} is defined to be $\displaystyle \sum_{v\in V} u^{deg^+(v)}$.
\end{defn}

The following theorem will be useful in this paper.

\begin{thm}[\cite{taniguchi2021quandle}]
    Let $D$ and $D'$ be oriented knot diagrams representing knots $L$ and $L'$, respectively. Let $P = p_1p_2 \cdots p_n$ be a positive integer for some distinct primes $p_1, \cdots , p_n$. For any $S \subset Hom(R_P,R_P)$, the quandle coloring quivers $Q^S_{R_P}
(L)$ and $Q^S_{R_P}
(L')$ are isomorphic if and only
if $|Hom(X_L,R_P)| = |Hom(X_{L'},R_P)|$. \label{thm:Taniguchi}
\end{thm}
\subsubsection{Column group enhancements}

We now discuss another type of enhancement that we will experiment with in the final section of the paper. Readers are encouraged to consult \cite{hennig2012column} for more details. As mentioned in the previous subsection, a quandle can be represented as a matrix \(M\). Due to the quandle rules, each column is a permutation. For instance, for the matrix in the previous section, column 1 is the permutation
\[
1 \mapsto 1,\quad 2 \mapsto 3,\quad 3 \mapsto 2.
\]

Now, treat each vertex as a quandle-colored diagram. The labels of that diagram generate a subquandle of the original quandle. Each element of such a subquandle corresponds to a column of \(M\) and hence to a permutation. We can then record the order of the subgroup of the permutation group generated by these columns. This yields a multiset of cardinality \(Col_X(L)\), where each entry records the order of the corresponding subgroup. As usual, a multi-set invariant can be turned into a polynomial invariant by making the elements the exponents of a variable $u$. The multiplicities are turned into coefficients. For instance, $\{0,1,1,2,2,2,3,3,3,3,3\}$ corresponds to the polynomial $1+2u+3u^2+5u^3$.

We will demonstrate this enhancement with an example in the final section.

\section{Proof of Theorem \ref{thm:main}}
We now prove our main result. Define $\kappa_{m,n}$ to be the connected sum of $m-1$ copies of $L$ and $n-m$ copies of $J$ as shown in Figure \ref{fig:k35}. We remark that for virtual knots, the connected sum depends on the location where the operation takes place. In Figure \ref{fig:k35}, we show the connected sum on $\kappa_{3,5},$ but for the general $\kappa_{m,n},$ we connect sum on a short arc that receives the color last in the Wirtinger coloring sequence of one summand with a seed of the other summand (see Subsection \ref{subsec:bridge} for a reminder of these definitions). This is to ensure that $b_1(K)$ does not increase when we do not want it to.

\main*
\begin{proof}
    \textbf{Claim 1:} $b_1(\kappa_{m-1,n})=m.$

To see that $b_1(\kappa_{m,n})\geq m.$ Observe that the virtual knot $L$ is tricolorable. Translating in terms of quandles, this means that $L$ admits nine colorings by the dihedral quandle of order 3. Using the connected sum formula found in \cite{clark2016quandle}, we see that $Col_{R_3}(\#^{m-1}L) = 3^{m}\leq 3^{b_1(L)}$.

On the other hand, $b_1(\kappa_{m-1,n})\leq m$ because there are $m$ seeds and one can use the main result in \cite{pongtanapaisan2019wirtinger} to search for a diagram for $\kappa_{m-1,n}$ with $m$ overpasses. The seeds are illustrated in Figure \ref{fig:k35} for a particular pair $(m,n)=(3,5),$ but in general, we need one seed from the $J$ summand (this seed is also shown in Figure \ref{fig:1seed}). In fact, this seed propagates through all the $J$ summands. Now, for each $L$ summand, we need one additional seed at a local maximum. This is illustrated in red and green in Figure \ref{fig:k35}. By the instruction of how we do connected sum (last strand that receives Wirtinger color of the previous summand connects to the first strand that receives Wirtinger color), we can be sure that $m$ seeds gives a complete Wirtinger coloring sequence.

 \textbf{Claim 2:} $b_2(\kappa_{m-1,n})=n.$

It is easier to see that $b_2(\kappa_{m-1,n})\leq n$ since we are stacking the connected sums vertically, making sure that $b_2$ increases by one for each summand. More precisely, each time the connected sum operation, we lose one local maximum and one local minimum, giving the classical Schubert bounds $b_2(K_1\# K_2)=b_2(K_1)+b_2(K_2)-1$  \cite{schultens2003additivity}.

For the other inequality, consider the biquandle 
\begin{align*}
    T=(((1,3,4,2), (3, 1,2,4), (2,4,3,1),(4,2,1,3)), \\ ((1,4,2,3), (2,3,1,4), (4,1,3,2),(3,2,4,1)))
\end{align*}

 There are 16 biquandle colorings for $J$, which are listed here:
\begin{table}[h!]
\centering
\begin{tabular}{|c|c|c|c|c|c|c|c|}
\hline
\textbf{a} & \textbf{b} & \textbf{c} & \textbf{d} & \textbf{e} & \textbf{f} & \textbf{g} & \textbf{h} \\
\hline
1 & 1 & 1 & 1 & 1 & 1 & 1 & 1 \\
1 & 2 & 1 & 3 & 2 & 2 & 2 & 3 \\
1 & 3 & 1 & 4 & 4 & 3 & 4 & 4 \\
1 & 4 & 1 & 2 & 3 & 4 & 3 & 2 \\
2 & 1 & 2 & 3 & 4 & 1 & 4 & 3 \\
2 & 2 & 2 & 1 & 3 & 2 & 3 & 1 \\
2 & 3 & 2 & 2 & 1 & 3 & 1 & 2 \\
2 & 4 & 2 & 4 & 2 & 4 & 2 & 4 \\
3 & 1 & 3 & 2 & 2 & 1 & 2 & 2 \\
3 & 2 & 3 & 4 & 1 & 2 & 1 & 4 \\
3 & 3 & 3 & 3 & 3 & 3 & 3 & 3 \\
3 & 4 & 3 & 1 & 4 & 4 & 4 & 1 \\
4 & 1 & 4 & 4 & 3 & 1 & 3 & 4 \\
4 & 2 & 4 & 2 & 4 & 2 & 4 & 2 \\
4 & 3 & 4 & 1 & 2 & 3 & 2 & 1 \\
4 & 4 & 4 & 3 & 1 & 4 & 1 & 3 \\
\hline
\end{tabular}
\caption{This table accompanies Figure \ref{fig:realspecial} (left).}
\label{tab:example}
\end{table}

 Note that $L$ is obtained from $J$ by performing oriented 2-moves twice. It turns out that the biquandle colorings by $T$ are preserved by this oriented move. Just in case, the biquandle colorings of $L$ are also listed in Table \ref{tab:example2}. A pattern that emerges in Table \ref{tab:example} and Table \ref{tab:example2} is that for any fixed short arc, the 16 colorings are partitioned into four blocks of size four. This means that for each coloring of one summand, there are four colored diagrams for the other summand such that the labels at the connected sum short arcs match up. This implies that $Col_T(K_1\#K_2) = \frac{1}{4}Col_T(K_1)\cdot Col_T(K_2) = 4^3$. Here, $K_i$ can be another copy of $J$ or a copy of $L$. By induction, we can conclude that $Col_T(\kappa_{m,n}\#K_1)=\frac{1}{4}4^n\cdot 16=4^{n+1} \leq 4^{b_2(\kappa_{m,n+1})}.$ Taking the log gives the desired result.
 
 \begin{table}[h!]
\centering
\begin{tabular}{|c|c|c|c|c|c|c|c|c|c|c|c|c|c|c|c|}
\hline
\textbf{a} & \textbf{b} & \textbf{c} & \textbf{d} & \textbf{e} & \textbf{f} & \textbf{g} & \textbf{h} & \textbf{i} & \textbf{j} & \textbf{k} & \textbf{l} & \textbf{m} & \textbf{n} & \textbf{o} & \textbf{p} \\
\hline
1 & 1 & 1 & 1 & 1 & 1 & 1 & 1 & 1 & 1 & 1 & 1 & 1 & 1 & 1 & 1 \\
1 & 2 & 1 & 2 & 1 & 4 & 3 & 2 & 3 & 2 & 3 & 4 & 3 & 4 & 1 & 4 \\
1 & 3 & 1 & 3 & 1 & 2 & 2 & 3 & 2 & 3 & 2 & 2 & 2 & 2 & 1 & 2 \\
1 & 4 & 1 & 4 & 1 & 3 & 4 & 4 & 4 & 4 & 4 & 3 & 4 & 3 & 1 & 3 \\
2 & 1 & 2 & 1 & 2 & 2 & 3 & 1 & 3 & 1 & 3 & 2 & 3 & 2 & 2 & 2 \\
2 & 2 & 2 & 2 & 2 & 3 & 1 & 2 & 1 & 2 & 1 & 3 & 1 & 3 & 2 & 3 \\
2 & 3 & 2 & 3 & 2 & 1 & 4 & 3 & 4 & 3 & 4 & 1 & 4 & 1 & 2 & 1 \\
2 & 4 & 2 & 4 & 2 & 4 & 2 & 4 & 2 & 4 & 2 & 4 & 2 & 4 & 2 & 4 \\
3 & 1 & 3 & 1 & 3 & 4 & 4 & 1 & 4 & 1 & 4 & 4 & 4 & 4 & 3 & 4 \\
3 & 2 & 3 & 2 & 3 & 1 & 2 & 2 & 2 & 2 & 2 & 1 & 2 & 1 & 3 & 1 \\
3 & 3 & 3 & 3 & 3 & 3 & 3 & 3 & 3 & 3 & 3 & 3 & 3 & 3 & 3 & 3 \\
3 & 4 & 3 & 4 & 3 & 2 & 1 & 4 & 1 & 4 & 1 & 2 & 1 & 2 & 3 & 2 \\
4 & 1 & 4 & 1 & 4 & 3 & 2 & 1 & 2 & 1 & 2 & 3 & 2 & 3 & 4 & 3 \\
4 & 2 & 4 & 2 & 4 & 2 & 4 & 2 & 4 & 2 & 4 & 2 & 4 & 2 & 4 & 2 \\
4 & 3 & 4 & 3 & 4 & 4 & 1 & 3 & 1 & 3 & 1 & 4 & 1 & 4 & 4 & 4 \\
4 & 4 & 4 & 4 & 4 & 1 & 3 & 4 & 3 & 4 & 3 & 1 & 3 & 1 & 4 & 1 \\
\hline
\end{tabular}
\caption{Biquandle colorings for $L$.}
\label{tab:example2}
\end{table}  
\end{proof}
\begin{figure}[ht!]
\labellist
\small\hair 2pt
\pinlabel $c$ at 140 173
\pinlabel $a$ at 139 443
\pinlabel $h$ at 1 403
\pinlabel $e$ at 249 443
\pinlabel $f$ at 409 443
\pinlabel $d$ at 125 -13
\pinlabel $b$ at 290 293
\pinlabel $g$ at 275 173
\pinlabel $a$ at 609 759
\pinlabel $b$ at 929 759
\endlabellist
\includegraphics[width=3cm]{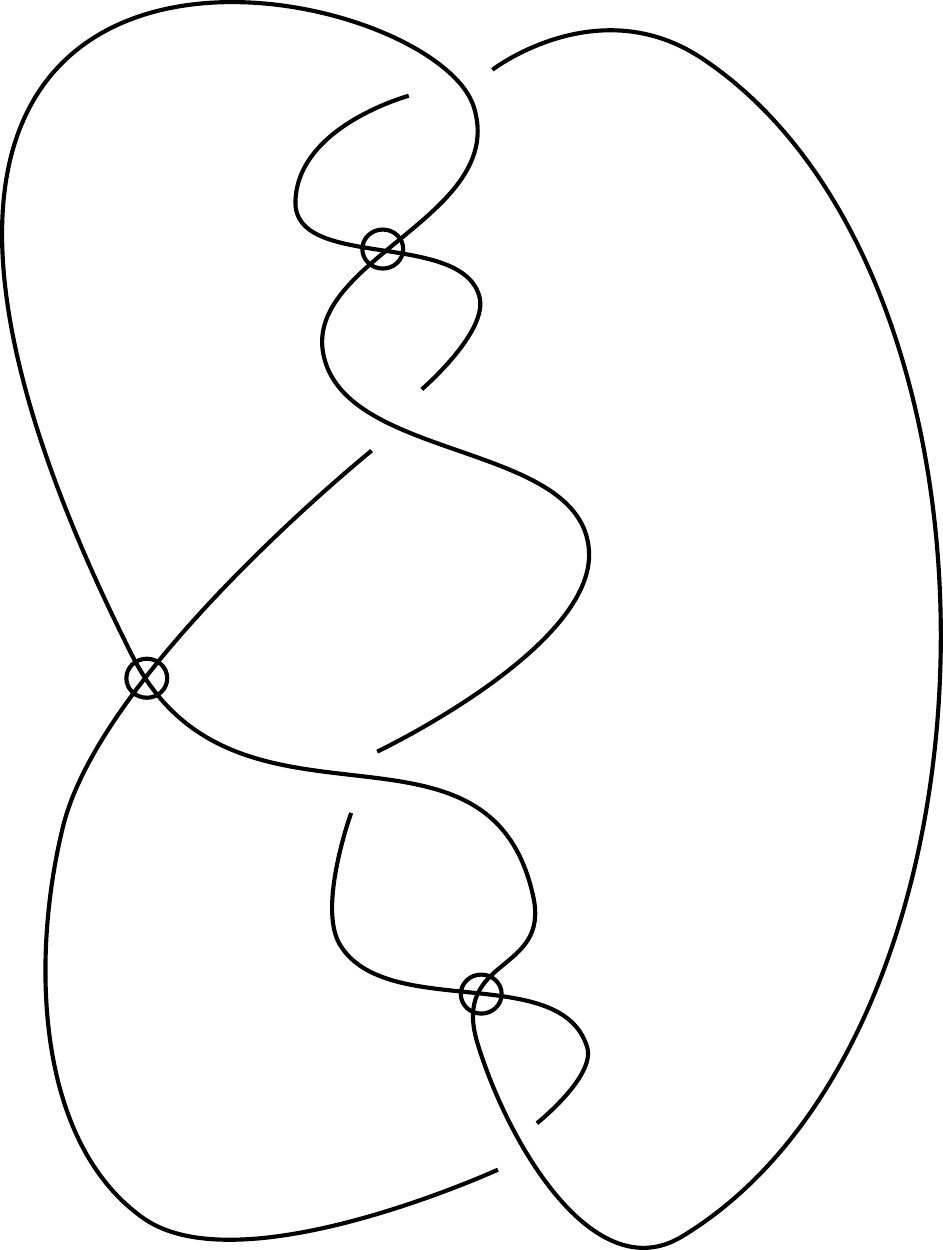}
\includegraphics[width=3cm]{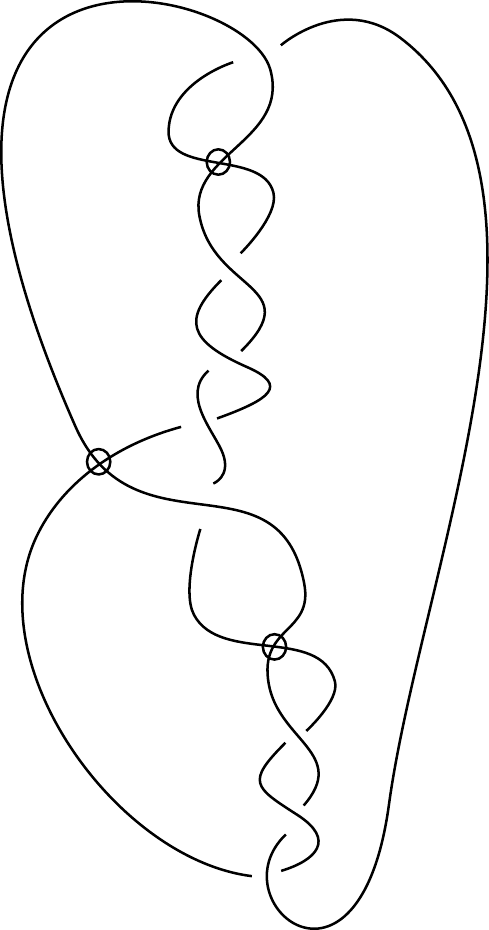}
\centering
\caption{(Left) The virtual knot $J$. The knot $J$ is oriented following the alphabetical order of the labels. (Right) The virtual knot $L$, which is also oriented following the alphabetical order of the labels. Not all labels are shown.}\label{fig:realspecial}
\end{figure}

\begin{figure}[ht!]
\includegraphics[width=5cm]{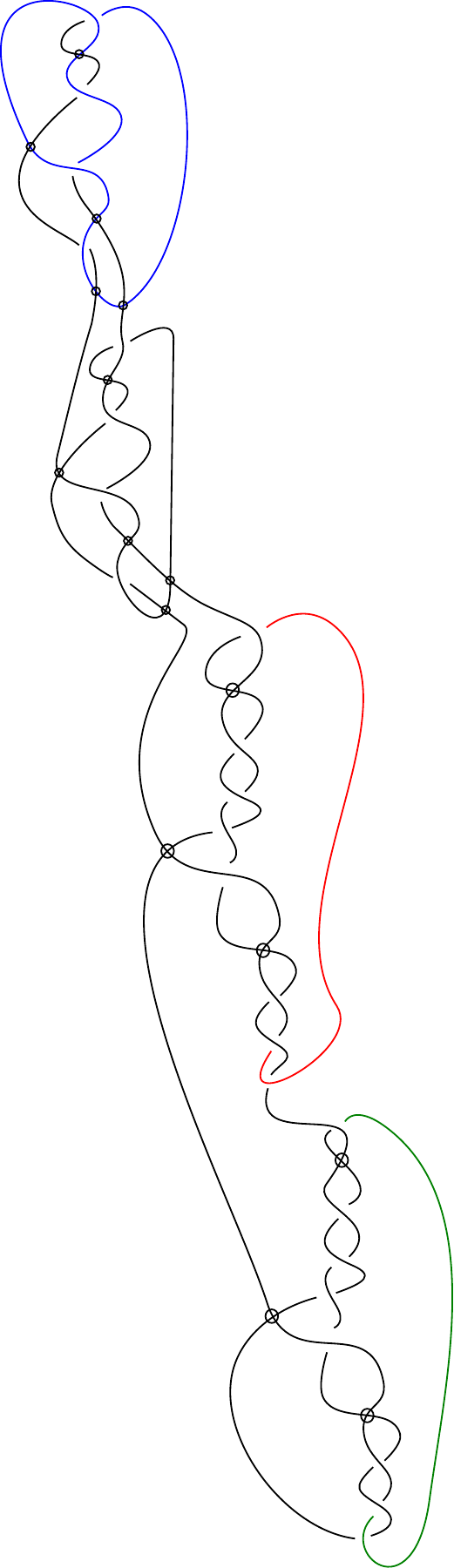}
\centering
\caption{The virtual knot $\kappa_{3,5}$ with the seeds shown.}\label{fig:k35}
\end{figure}

\section{Proper enhancements}

Some choices of quandles listed in Theorem \ref{thm:Taniguchi} will never give a proper enhancement. We mention another situation where the quivers are isomorphic.

\bridgequiver*

\begin{proof}
    For $i=1,2,$ there exist diagrams $D_i$ representing $L_i$, where $|X|^{b_2(L_i)}$ possible coloring assignments appear on short arcs containing the $b_2(L_i)$ local maxima. Thus, vertices of \( \mathcal{Q}^S_X(L_i) \) are indexed by $(\alpha_1^i,\alpha_2^i,\cdots, \alpha_{b_2(L_i)}^i) \in (\mathbb{Z}_{|X|})^{b_2(L_i)}$ according to the colors at these maxima. We define a bijective map $\varphi$ from vertices of \( \mathcal{Q}^S_X(L_1) \) to \( \mathcal{Q}^S_X(L_2) \) taking a vertex corresponding to $(\alpha_1^1,\alpha_2^1,\cdots, \alpha_{b_2(L_1)}^1)$ to a vertex corresponding to $(\alpha_1^2,\alpha_2^2,\cdots, \alpha_{b_2(L_2)}^2)$. For any edge connecting two elements of $(\mathbb{Z}_{|X|})^{b_2(L_1)}$, there is an edge  connecting two elements of $(\mathbb{Z}_{|X|})^{b_2(L_2)}$.
\end{proof}

Next, we will give infinitely many pairs of links, where the enhancement is proper.

Consider the biquandle $Z$ defined by $x\text{ }\overline{\triangleright}\text{ } y=3x$ and $x \text{ }\underline{\triangleright} \text{ }y=x+2y$ in $\mathbb{Z}_4$.

    \begin{center}
    \includegraphics[width=0.3\linewidth]{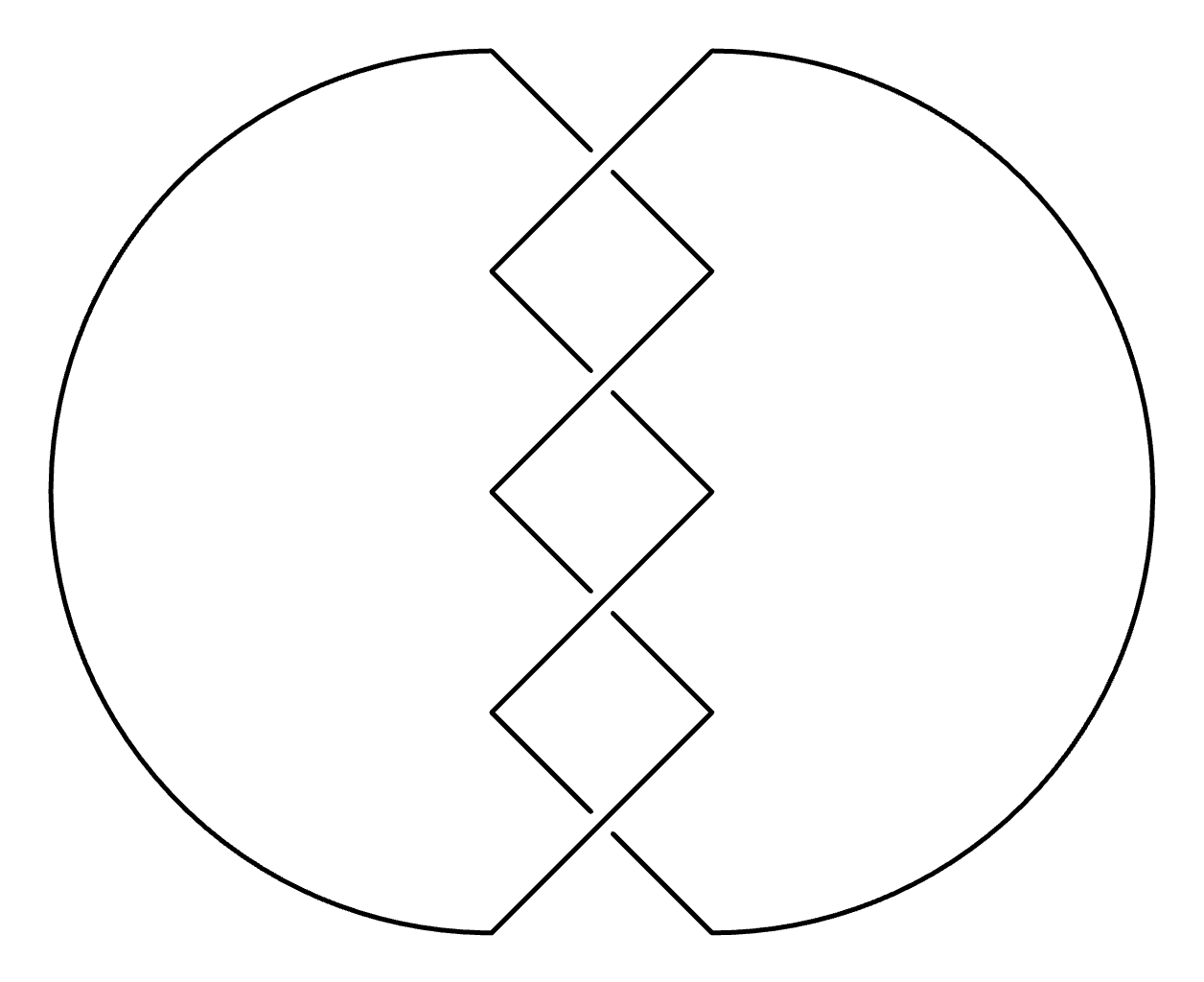}
    \label{fig:enter-label}
    \end{center}

For $T(2,4)$, there are 16 possible colorings from the following matrix:

\begin{center}
    $\begin{bmatrix}
    0 & 3 & 3 & 0 & 0 & 0 & 0 & 0\\
    0 & 0 & 0 & 3 & 3 & 0 & 0 & 0\\
    0 & 0 & 0 & 0 & 0 & 3 & 3 & 0\\
    3 & 0 & 0 & 0 & 0 & 0 & 0 & 3\\
    1 & 0 & 2 & 3 & 0 & 0 & 0 & 0\\
    0 & 0 & 1 & 0 & 2 & 3 & 0 & 0\\
    0 & 0 & 0 & 0 & 1 & 0 & 2 & 3\\
    2 & 3 & 0 & 0 & 0 & 0 & 1 & 0\\
    \end{bmatrix}$
\end{center}

With computer aid, we can solve for the null space of this matrix yielding the following colorings:
$$(0,0,0,0,0,0,0,0),(0,1,3,2,2,3,1,0),(0,2,2,0,0,2,2,0),(0,3,1,2,2,1,3,0),$$
$$(1,0,0,1,3,2,2,3),(1,1,3,3,1,1,3,3),(1,2,2,1,3,0,0,3),(1,3,1,3,1,3,1,3),$$
$$(2,0,0,2,2,0,0,2),(2,1,3,0,0,3,1,2),(2,2,2,2,2,2,2,2),(2,3,1,0,0,1,3,2),$$
$$(3,0,0,3,1,2,2,1),(3,1,3,1,3,1,3,1),(3,2,2,3,1,0,0,1),(3,3,1,1,3,3,1,1).$$

We can express the specified biquandle colorings of $T(2,n)$ when $n$ is a multiple of $4$:

\begin{lem}
     For all $n$ such that $n=4k$ for some $k\in\mathbb{Z}^+$, then $T(2,n)$ will have 16 colorings with respect to the biquandle $Z$.\label{lem:coloringt2n}
\end{lem}

 \begin{proof}
     Since the colorings of the maxima of $T(2,4)$ are independent of one another, and the minima of $T(2,4)$ always match the corresponding maxima, then the colorings are  uniquely determined by the initial choices.

     Then, $T(2,4)$ will have the maximum colorings of $4\cdot 4=16$ (four choices for each maximum), as shown above.

     Notice that $T(2,8)$ contains two $T(2,4)$ links attached together. In general, $T(2,4k)$ has a diagram consisting of $k$ copies of $T(2,4)$ links stacked vertically.

     Then the number of colorings hold for any $T(2,n)$, where $n$ is a multiple of 4 since any initial choice for the two local maxima determine the coloring of the entire link uniquely (the coloring is periodic after four crossings).
 \end{proof}

Next, we present links with the same coloring number, but different quiver invariant. In particular, the following link has 16 colorings

We outline the idea of finding an infinite family of pairs of links whose quandle coloring quivers are proper enhancements. Fix a positive integer $b>1$ and the biquandle $Z$ from Lemma \ref{lem:coloringt2n}. First, we find a link $L$ with $\beta(L) = b$ such that the bound in Proposition \ref{prop:bridge} is sharp, i.e.  $Col_X(L) = 4^b$. For example, the connected sum of $b-1$ copies of $T(2,4)$ will work.

Then, we let $L'$ be the chain link of $2b-1$ components (see Figure \ref{fig:chain}). We observe that $\beta(L') = 2b-1 >\beta(L)$. We will show that 
$Col_X(L') = 4^b =Col_X(L')$ so that the bound in Proposition \ref{prop:bridge} for $L'$ is not sharp.
Finally, we will show that their quandle coloring quivers are different.


\begin{figure}[ht!]
\includegraphics[width=3cm]{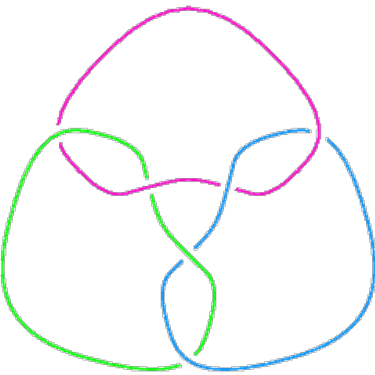}
\includegraphics[width=3cm]{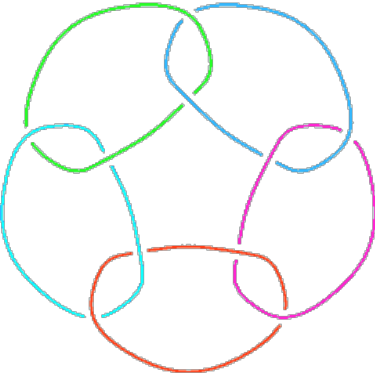}
\centering
\caption{From left to right, chain links of 3 and 5 components, respectively.}\label{fig:chain}
\end{figure}
\begin{lem}
    $Col_X(L')=4^{b}$\label{lem:samecount}
\end{lem}
\begin{proof}

Let $x_k$ be the label of the outer strand and $y_k$ be the label of the inner strand of the $k$-th component of $L'$ (Figure \ref{fig:2tangle}). From the dihedral quandle relation on each crossing on the $k$-th component, we have the following equations in $\mathbb{Z}/4\mathbb{Z}$:

\begin{align}
2 x_{k-1} &= x_k + y_k, \label{lem51} \\
2 y_k &= x_{k-1} + y_{k-1}, \label{lem52} \\
2 x_k &= x_{k+1} + y_{k+1}, \label{lem53} \\
2 y_{k+1} &= x_k + y_k, \label{lem54}
\end{align}
for each $k=1, \ldots , 2b -1$ and we can identify the $k$-th component with the $j$-th component when $k = j \pmod{2b-1}$ .

From the above equations, we see that $x_k + y_k$ is even, i.e. $x_k$ and $y_k$ must have the same parity.
From equations (\ref{lem51}) and (\ref{lem54}), we also have $2 x_{k-1} = 2 y_{k+1}$, which implies that both the strands in the $(k-1)$-th component and both the strands in the $(k+1)$-th components have the same parity. We then see that $x_1 , x_3 , \ldots, x_{2b-1}$ have the same parity. Moreover, it follows that, at $k=1$,  $x_2$ and $x_{2b-1}$ have the same parity. Consequently, every label must have the same parity.

It will also turn out that $x_k = y_k$ for every $k=1, \ldots , 2b -1$. Without loss of generality, we consider 2 cases $(x_1 , y_1) = (0,2)$ and $(x_1 , y_1) = (1,3)$. In the first case, we have $2 y_2 = 2$ and $y_2$ must be odd. In the second case, we have $2 y_2 = 0$ and $y_2$ must be even. In either case, this violates the same parity condition we deduced.

To count the number of all colorings, we first pick the parity (even/odd) of all labels. For each parity, there are exactly 2 choices for each $x_k = y_k$ for each $k=1, \ldots , 2b -1$. Thus, there are $2 \cdot 2^{2b-1} = 4^b$ colorings.






\end{proof}

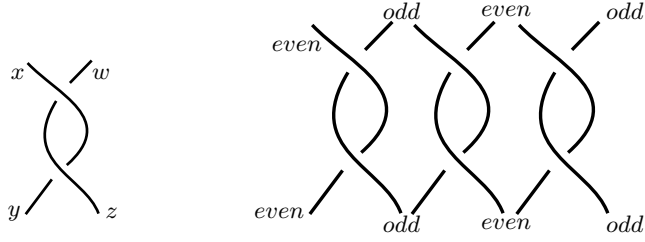
\begin{figure}[ht!]
\labellist
\small\hair 2pt
\pinlabel $x$ at -10 203
\pinlabel $y$ at -15 3
\pinlabel $w$ at 110 203
\pinlabel $z$ at 125 3
\pinlabel $even$ at 390 243
\pinlabel $even$ at 365 3
\pinlabel $odd$ at 540 293
\pinlabel $odd$ at 540 -13
\pinlabel $even$ at 690 293
\pinlabel $even$ at 690 -13
\pinlabel $odd$ at 860 293
\pinlabel $odd$ at 860 -13
\endlabellist
\includegraphics[width=1cm]{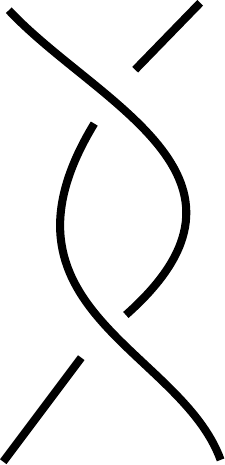}
\hspace{1in}
\includegraphics[width=4cm]{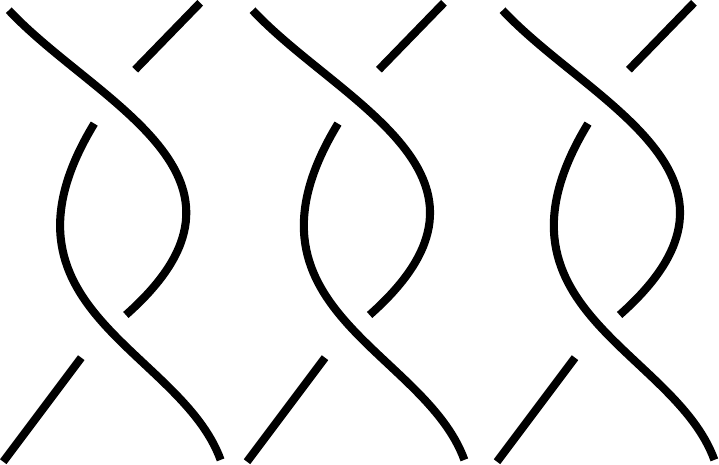}
\centering
\caption{(Left) Given this label, we have that $x+y$ is even and if $x\neq y$ then the other endpoints are distinct and opposite parity. (Right) If $x\neq y$, then the parities do not match up as we try to glue the right endpoints of $T_{2\beta(L)-1}$ to the left endpoints of $T_1$ and we do not have a consistent coloring.}\label{fig:2tangle}
\end{figure}
\begin{rem}
    The argument in Proposition \ref{lem:samecount} does not hold in general. That is, some labelings where the 4 endpoints do not have the same color can fit in a consistently colored diagram by a quandle. For instance, the readers should work out the coloring by $R_9$ of the pretzel knot $P(3,3,3)$ to see this phenomenon.
\end{rem}

\begin{theorem}\label{thm:1}
    The quiver invariant $\mathcal{Q}^{\varphi}_{X}$ with respect to the endomorphism $\varphi$ defined by $1\mapsto 2, 2\mapsto 4, 3\mapsto 2, 4\mapsto 4$ distinguishes $L$ from $L'$.
\end{theorem}

\begin{rem}
    By observation in \cite{khandhawit2024quandle}, an endomorphism of $R_n$ has a general form $x \mapsto a x + b$. We see that the above endomorphism is given by $\varphi (x) = 2x $ .
\end{rem}

\begin{proof}

For the link $L$, a coloring is completely determined by specifying labels of the $b$ strands of a bridge diagram of $L$ since the bound in Proposition \ref{prop:bridge} is sharp. Thus, we can identify a set of vertices of $\mathcal{Q}^{\varphi}_{X}(L)$ with a set of tuples 
$$ V_{\mathcal{Q}^{\varphi}_{X}(L)} = \{ (x_1 , \ldots , x_b ) \,|\, x_k \in \mathbb{Z}/4\mathbb{Z} \}.$$
For $x, x' \in V_{\mathcal{Q}^{\varphi}_{X}(L)}$, we see that there is an edge from $x$ to $x'$ when $x'_k = 2 x_k$ for every $k$. We observe that $y' = 2y \pmod{4}$ has exactly 2 solutions when $y'$ is even and there is no solution otherwise. Thus, a vertex $x'$ has incoming edges if and only if all $x'_k$ is even. There are exactly such $2^b$ vertices and all of them have in-degree $2^b$.

Next, we consider the quiver of the link $L'$.
From the proof of Proposition \ref{lem:samecount}, we can identifty the set of vertices as
$$ V_{\mathcal{Q}^{\varphi}_{X}(L')} = \{ (x_1 , \ldots , x_{2b-1} ) \,|\, x_k \in \mathbb{Z}/4\mathbb{Z} \text{ and all } x_k \text{ have the same parity} \}.$$
Suppose that there is an edge in $\mathcal{Q}^{\varphi}_{X}(L')$  from $x$ to $x'$. Similarly, we have $x'_k = 2 x_k$ for every $k$, so each $x'_k$ is even. If there were $j \ne k$ such that $x'_j = 0$ and $x'_k = 2$, we would have $x_j$ even and $x_k$ odd and they would fail to have the same parity. Therefore, there are exactly 2 vertices, namely $(0,\ldots,0)$ and $(2,\ldots,2)$, which have incoming edges. Moreover, their in-degree are $2^{2b-1}$. 

Therefore, the quivers $\mathcal{Q}^{\varphi}_{X}(L)$ and $\mathcal{Q}^{\varphi}_{X}(L')$ are not isomorphic and they are proper enhancement of the coloring invariants.


\end{proof}

Although the bridge indices of the pairs we constructed in the previous section can be arbitrary, they are different.  
In this section, we will construct other infinitely many pairs of 3-bridge links that give proper enhancement. In particular, the bound in Proposition \ref{prop:bridge} for these links will not be sharp.  

Let $p$ be an odd prime number and $r$ be an integer relatively prime to $p$. We will use the dihedral quandle $X = R_{p^2}$ of order $p^2$.  The links we will consider are the pretzel link $P(p^2,2r,p^2),$ and
the connected sum $T(2,p)\# T(2,p)$ of the 2-bridge torus knots.




\begin{lem} \label{lemmpretz}
    $Col_{R_{p^2}}(P(p^2,2r,p^2)) = p^4$.
\end{lem}
\begin{proof}
    In the standard pretzel diagram, there are three local maxima with respect to the standard height function on $\mathbb{R}^2$. We will assign labelings $x,y,z$ to the strands $s_1,s_2,s_3$ containing these local maxima.

    Observe that the labelings of every strand in the twist region is determined by the labeling of the northwest strand and the northeast strand. In particular, when the labelings are $a$ and $b$ and the region has $m$ twists, the dihedral quandle relation gives
    \[ c = mb - (m-1)a  \text{ and } d = (m+1)b - ma, \]
    where $c$ and $d$ are the labelings of the southwest strand and the southeast strand respectively.
    In the special case where $m = p^2$, we see that $c=a$ and $d=b$.

    Consequently, the labeling $(x,y,z)$ of $s_1,s_2,s_3$ propagate labelings to the entire diagram by the quandle rules (see Figure \ref{fig:pretz}). However, there is one constraint given by the middle twist region, i.e. $y = 2rz - (2r-1)y$. We see that the constraint is equivalent to $0 = 2r(z-y) \pmod{p^2}$ and happens only when $y=z$ since $gcd(p^2,2r)=1$.

    We can conclude that the coloring of $P(p^2,2r,p^2)$ is completely determined by $x$ and $y$ and, thus, $Col_{R_{p^2}}(P(p^2,2r,p^2)) = p^2\cdot p^2 = p^4.$


    
\end{proof}

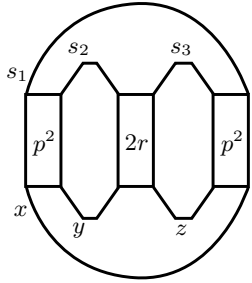
\begin{figure}[ht!]
\labellist
\small\hair 2pt
\pinlabel $p^2$ at 12 80
\pinlabel $2r$ at 66 80
\pinlabel $p^2$ at 122 80
\pinlabel $s_1$ at -4 120
\pinlabel $x$ at -2 40
\pinlabel $s_2$ at 32 135
\pinlabel $y$ at 32 29
\pinlabel $s_3$ at 92 135
\pinlabel $z$ at 92 29
\endlabellist
\includegraphics[width=3cm]{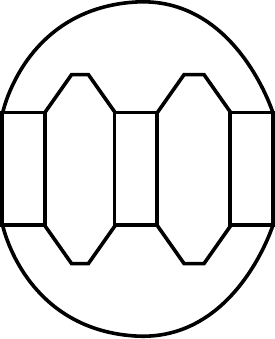}
\centering
\caption{Labeling pretzel links.}\label{fig:pretz}
\end{figure}


\begin{lem}
    $Col_{R_{p^2}}(T(2,p)\# T(2,p)) = p^4$.
\end{lem}

\begin{proof}
    By Proposition \ref{prop:propagate}, we have $Col_{R_{p^2}}(T(2,p))=p^2 \cdot gcd(p^2,p) = p^3.$ From the fact that the dihedral quandle of odd order is a faithful quandle \cite{bardakov2022general} and the formula of quandle counting invariants under connected sum  in \cite{clark2016quandle}, we obtain  $Col_{R_{p^2}}(T(2,p)\# T(2,p)) = (p^3\cdot p^3)/p^2=p^4$ as desired.
\end{proof}

We consider an endomorphism $\varphi$ of $R_{p^2}$ given by $\varphi(x) = px$.

\begin{thm}
     The quiver invariant $\mathcal{Q}^{\varphi}_{R_{p^2}}$  distinguishes $T(2,p)\# T(2,p)$ from $P(p^2,2r,p^2)$.
\end{thm}

\begin{proof}
    From the proof of Lemma \ref{lemmpretz}, we see that the set of vertices of $\mathcal{Q}^{\varphi}_{R_{p^2}}(P(p^2,2r,p^2))$
    can be identified with the set
    \[ V = \{ (x , y ) \,|\, x,y \in \mathbb{Z}/p^2\mathbb{Z} \} \]
    We see that there is an edge from $v = (x,y)$ to $v' = (x',y')$ if and only if $x' = px \pmod{p^2}$ and $y' = py \pmod{p^2}$. 
    In particular, the vertex $v'$ has an incoming edge when both $x'$ and $y'$ are divisible by $p$. Moreover, when $x' = pa'$, there are exactly $p$ solutions for $x$ solving $pa' = px \pmod{p^2}$. We can deduce that there are exactly $p^2$ vertices with incoming edges and each of them has in-degree $p^2$.

   On the other hand, we know how to precisely describe all the colorings of $T(2,p)$  by $R_{p^2}$ (c.f. \cite{elhamdadi2025quandle}).  
   In particular, we can identify them with the set
   $$\{ (x , y ) \in \mathbb{Z}/p^2\mathbb{Z} \times \mathbb{Z}/p^2\mathbb{Z}   \,|\, px = py \pmod{p^2} \}. $$
   Note that the condition $px = py \pmod{p^2}$ is equivalent to $x = y \pmod{p}$.

   For the connected sum $T(2,p)\# T(2,p)$, we can analogously identify the colorings of another summand as the pair $(z,w)$ with $z = w \pmod{p}$. By connecting the strand $y$ with the strand $w$, we can describe the set of vertices of $\mathcal{Q}^{\varphi}_{R_{p^2}}(T(2,p)\# T(2,p))$ by the set
   $$ \{ (x , y , z) \in \mathbb{Z}/p^2\mathbb{Z} \times \mathbb{Z}/p^2\mathbb{Z} \times \mathbb{Z}/p^2\mathbb{Z}   \,|\, px = py = pz \pmod{p^2} \} $$
   We see that there is an edge from $v = (x,y,z)$ to $v' = (x',y',z')$ when $x' = px$, $y'=py$, and $z'=pz$. 
   Since $px = py = pz \pmod{p^2}$, this forces $x'=y'=z'$ and $p \mid x'$. Thus, there are exactly $p$ vertices with incoming edges.
   Since there are exactly $p$ solutions for $x$ solving $pa' = px \pmod{p^2}$, each of those $p$ vertices has in-degree $p^3$.

   Therefore, the quivers $\mathcal{Q}^{\varphi}_{R_{p^2}}(P(p^2,2r,p^2)))$ and $\mathcal{Q}^{\varphi}_{R_{p^2}}(T(2,p)\# T(2,p))$ are not isomorphic and they are proper enhancement of the coloring invariants.  
   


\end{proof}

\section{Generalizations, discussions, and experimentations}

One can use a simple trick to obtain proper enhancements for knotted surfaces in 4-space. There is a standard construction of knotted 2-spheres in 4-space from classical knots in 3-space. This construction due to Zeeman is called \textit{twist-spinning} \cite{zeeman1965twisting}.

Take a knot in 3-space and remove a 1-string trivial tangle from it. This results in another tangle $(B^3,T).$ Then, $(B^3,T)\times I$ is a disk inside a 4-ball. Take two copies of this disk tangle, and glue them together with a diffeomorphism that rotates the knotted arc around its axis $m$ times.

The authors of \cite{sato2022bridge} argued that the number of colorings by a dihedral quandle of a classical knot is exactly the same as the number of colorings of its $2k$-twist spun. Therefore, examples in the previous section can be transported one dimension up to give examples of proper enhancements.

We end the paper with results in similar spirit with \cite{cho2019quandle2, istanbouli2020quandle}, but we use the column group enhancement instead of the quandle cocycle invariant or the quandle module invariant. Consider the following matrix for a quandle.
\[
\begin{bmatrix}
1 & 3 & 5 & 7 & 9 & 2 & 4 & 6 & 8 \\
9 & 2 & 4 & 6 & 8 & 1 & 3 & 5 & 7 \\
8 & 1 & 3 & 5 & 7 & 9 & 2 & 4 & 6 \\
7 & 9 & 2 & 4 & 6 & 8 & 1 & 3 & 5 \\
6 & 8 & 1 & 3 & 5 & 7 & 9 & 2 & 4 \\
5 & 7 & 9 & 2 & 4 & 6 & 8 & 1 & 3 \\
4 & 6 & 8 & 1 & 3 & 5 & 7 & 9 & 2 \\
3 & 5 & 7 & 9 & 2 & 4 & 6 & 8 & 1 \\
2 & 4 & 6 & 8 & 1 & 3 & 5 & 7 & 9
\end{bmatrix}
\]

Consider the knots \(9_{24}\) and \(6_1\). One can check that while the column enhancements for these knots are identical, their quiver enhancements are distinct.

The knot \(9_{24}\) has 9 trivial labelings. Since each column generates a subgroup of order 2, this contributes the term \(9u^2\) to our polynomial. Some colored diagrams have labelings that generate the subquandle \(\{1,4,7\}\). Now, consider columns 1, 4, and 7 of the matrix above. The subgroup of the permutation group generated by these columns has order 6, and there are 6 colored diagrams associated with the \(\{1,4,7\}\) subquandle. Similarly, one can verify that there are 6 colored diagrams associated with the \(\{2,5,8\}\) subquandle and 6 colored diagrams associated with the \(\{3,6,9\}\) subquandle. This contributes a term \(18u^6\) to the polynomial. Lastly, all other 54 colored diagrams generate a subgroup of order 18. In conclusion, the column group enhancement polynomial is
\[
54u^{18} + 18u^6 + 9u^2.
\]

We encourage the reader to check that the column group enhancement for \(6_1\) is also \(54u^{18} + 18u^6 + 9u^2\). Now, consider an endomorphism \(f(x)=3x\). With the aid of computers, we can verify that the in-degree polynomial of \(\mathcal{Q}^{f}_{R_{9}}(6_1)\) does not contain a \(u^3\) term, whereas \(\mathcal{Q}^{f}_{R_{9}}(9_{24})\) does.

Recently, Nelson introduced an intriguing quiver representation–theoretic invariant of links and virtual links \cite{nelson2024quandle}. It would be interesting to construct infinitely many links to show that the quiver representation invariant is stronger than the quandle coloring quiver.

\bibliographystyle{amsplain}
\bibliography{ref}

\providecommand{\bysame}{\leavevmode\hbox to3em{\hrulefill}\thinspace}
\providecommand{\MR}{\relax\ifhmode\unskip\space\fi MR }
\providecommand{\MRhref}[2]{%
  \href{http://www.ams.org/mathscinet-getitem?mr=#1}{#2}
}
\providecommand{\href}[2]{#2}
\begin{thebibliography}{10}

\bibitem{bardakov2022general}
Valeriy Bardakov, Timur Nasybullov, and Mahender Singh, \emph{General constructions of biquandles and their symmetries}, Journal of Pure and Applied Algebra \textbf{226} (2022), no.~7, 106936.

\bibitem{blair2020wirtinger}
R~Blair, A~Kjuchukova, R~Velazquez, and P~Villanueva, \emph{Wirtinger systems of generators of knot groups}, Communications in Analysis and Geometry \textbf{28} (2020), no.~2, 243--262.

\bibitem{cho2019quandle2}
Karina Cho and Sam Nelson, \emph{Quandle cocycle quivers}, Topology and its Applications \textbf{268} (2019), 106908.

\bibitem{cho2019quandle}
\bysame, \emph{Quandle coloring quivers}, Journal of Knot Theory and Its Ramifications \textbf{28} (2019), no.~01, 1950001.

\bibitem{clark2016quandle}
W~Edwin Clark, Masahico Saito, and Leandro Vendramin, \emph{Quandle coloring and cocycle invariants of composite knots and abelian extensions}, Journal of knot theory and its ramifications \textbf{25} (2016), no.~05, 1650024.

\bibitem{elhamdadi2025quandle}
Mohamed Elhamdadi, Brooke Jones, and Minghui Liu, \emph{Quandle coloring quivers of general torus links by dihedral quandles}, Journal of Algebraic Combinatorics \textbf{61} (2025), no.~1, 1--17.

\bibitem{elhamdadi2015quandles}
Mohamed Elhamdadi and Sam Nelson, \emph{Quandles}, vol.~74, American Mathematical Soc., 2015.

\bibitem{fenn2004biquandles}
Roger Fenn, Mercedes Jordan-Santana, and Louis Kauffman, \emph{Biquandles and virtual links}, Topology and its Applications \textbf{145} (2004), no.~1-3, 157--175.

\bibitem{hennig2012column}
Johanna Hennig and Sam Nelson, \emph{The column group and its link invariants}, Journal of Knot Theory and Its Ramifications \textbf{21} (2012), no.~07, 1250063.

\bibitem{istanbouli2020quandle}
Karma Istanbouli and Sam Nelson, \emph{Quandle module quivers}, Journal of Knot Theory and Its Ramifications \textbf{29} (2020), no.~12, 2050084.

\bibitem{joyce1982classifying}
David Joyce, \emph{A classifying invariant of knots, the knot quandle}, Journal of Pure and Applied Algebra \textbf{23} (1982), no.~1, 37--65.

\bibitem{khandhawit2024quandle}
Tirasan Khandhawit, Korn Kruaykitanon, and Puttipong Pongtanapaisan, \emph{Quandle coloring quivers and 2-bridge links}, European Journal of Mathematics \textbf{10} (2024), no.~2, 39.

\bibitem{kim2021quandle}
Jieon Kim, Sam Nelson, and Minju Seo, \emph{Quandle coloring quivers of surface-links}, Journal of Knot Theory and Its Ramifications \textbf{30} (2021), no.~01, 2150002.

\bibitem{matveev1982distributive}
Sergei~Vladimirovich Matveev, \emph{Distributive groupoids in knot theory}, Matematicheskii Sbornik \textbf{161} (1982), no.~1, 78--88.

\bibitem{nakanishi2015two}
Yasutaka Nakanishi and Shin Satoh, \emph{Two definitions of the bridge index of a welded knot}, Topology and its Applications \textbf{196} (2015), 846--851.

\bibitem{nelson2024quandle}
Sam Nelson, \emph{Quandle cohomology quiver representations}, arXiv preprint arXiv:2411.02153 (2024).

\bibitem{pongtanapaisan2019wirtinger}
Puttipong Pongtanapaisan, \emph{Wirtinger numbers for virtual links}, Journal of Knot Theory and Its Ramifications \textbf{28} (2019), no.~14, 1950086.

\bibitem{sato2022bridge}
Kouki Sato and Kokoro Tanaka, \emph{The bridge number of surface links and kei colorings}, Bulletin of the London Mathematical Society \textbf{54} (2022), no.~5, 1763--1771.

\bibitem{schultens2003additivity}
Jennifer Schultens, \emph{Additivity of bridge numbers of knots}, Mathematical Proceedings of the Cambridge Philosophical Society, vol. 135, Cambridge University Press, 2003, pp.~539--544.

\bibitem{taniguchi2021quandle}
Yuta Taniguchi, \emph{Quandle coloring quivers of links using dihedral quandles}, Journal of Knot Theory and Its Ramifications \textbf{30} (2021), no.~02, 2150011.

\bibitem{vo2024learning}
Hanh Vo, Puttipong Pongtanapaisan, and Thieu Nguyen, \emph{Learning bridge numbers of knots}, arXiv preprint arXiv:2405.05272 (2024).

\bibitem{zeeman1965twisting}
E~Christopher Zeeman, \emph{Twisting spun knots}, Transactions of the American Mathematical Society \textbf{115} (1965), 471--495.

\end{thebibliography}


\end{document}